\documentclass[manyauthors]{fundam}

\publyear{25}
\papernumber{1}
\volume{194}
\issue{2}
\theDOI{10.46298/fi.12374}

\versionForARXIV

\newcommand{\F}{F}

\newcommand{\maps}{\colon}
\newcommand{\unfold}{\mathsf{unfold}}

\newcommand{\op}{\mathrm{op}}
\newcommand{\namedset}[1]{\mathbb{#1}}
\newcommand{\N}{\namedset N}

\newcommand{\inc}{\mathsf{in}}
\newcommand{\outc}{\mathsf{out}}

\newcommand{\Kl}{\mathsf{Kl}}
\newcommand{\namedcat}[1]{\mathsf{#1}}
\newcommand{\C}{\namedcat{C}}
\newcommand{\D}{\namedcat{D}}

\newcommand{\Set}{\namedcat{Set}}

\newcommand{\Hom}{\mathrm{Hom}}

\newcommand{\Coeq}{\mathsf{Coeq}}

\newcommand{\Rel}{\mathsf{Rel}}

\newcommand{\colim}{\mathrm{colim}}
\newcommand{\id}{\text{id}}
\newcommand{\Hylo}[1][F]{\mathsf{CA}_{#1}}
\newcommand{\str}{\mathsf{cart}}
\newcommand{\FCoalg}[1][F]{\mathsf{Coalg}_{#1}}
\newcommand{\FAlg}[1][F]{\mathsf{Alg}_{#1}}
\newcommand{\nexttime}{\fullmoon}
\newcommand{\Ar}{\mathsf{ar}}
\newcommand{\MS}{\mathsf{MS}}
\newcommand{\SqMS}{\mathsf{SqMS}}
\makeatletter
\newcommand*{\relrelbarsep}{.386ex}
\newcommand*{\relrelbar}{%
  \mathrel{%
    \mathpalette\@relrelbar\relrelbarsep
  }%
}
\newcommand*{\@relrelbar}[2]{%
  \raise#2\hbox to 0pt{$\m@th#1\relbar$\hss}%
  \lower#2\hbox{$\m@th#1\relbar$}%
}
\providecommand*{\rightrightarrowsfill@}{%
  \arrowfill@\relrelbar\relrelbar\rightrightarrows
}
\providecommand*{\leftleftarrowsfill@}{%
  \arrowfill@\leftleftarrows\relrelbar\relrelbar
}
\providecommand*{\xrightrightarrows}[2][]{%
  \ext@arrow 0359\rightrightarrowsfill@{#1}{#2}%
}
\providecommand*{\xleftleftarrows}[2][]{%
  \ext@arrow 3095\leftleftarrowsfill@{#1}{#2}%
}
\makeatother

\newenvironment{prfsketch}
  {\trivlist\PRstyle\item[]{\bfseries Proof [Sketch]:}\newline}{\QED\endtrivlist}
\def\squareforqed{\hbox{\rlap{$\sqcap$}$\sqcup$}}
\def\QED{\ifmmode\squareforqed\else{\unskip\nobreak\hfil
\penalty50\hskip1em\null\nobreak\hfil\squareforqed
\parfillskip=0pt\finalhyphendemerits=0\endgraf}\fi}

\usepackage{amssymb}
\usepackage{tikz}
\usepackage{enumitem}
\usepackage{wasysym}

\usepackage{comment}

\usepackage{graphicx}
\usepackage{adjustbox}
\usepackage[all,2cell]{xy}\UseAllTwocells\SilentMatrices
\definecolor{darkgreen}{rgb}{0,0.45,0}
\definecolor{myurlcolor}{rgb}{0,.45,.2}
\definecolor{mylinkcolor}{rgb}{.1,.1,.7}

\newcommand{\R}{\mathbb{R}}

\usepackage{tikz}
\usepackage{tikz-cd}
\usetikzlibrary{backgrounds,circuits,circuits.ee.IEC,shapes,fit,matrix}
\tikzset{pullback/.style={minimum size=1.2ex,path picture={
\draw[opacity=1,black, -] (-0.5ex,-0.5ex) -- (0.5ex,-0.5ex) -- (0.5ex,0.5ex);%
}}}

\usepackage{verbatim}
\usetikzlibrary{arrows,shapes}

\definecolor{joecolor(x11)}{rgb}{0.0, 0.5, 0.5}

\definecolor{purple(x11)}{rgb}{0.8, 0, 0.8}

\usepackage[strings]{underscore}

\begin{document}

\title{Relative fixed points of functors} 


\author{Ezra Schoen\corresponding \\
Strathclyde University, United Kingdom \\ ezra.schoen@strath.ac.uk 
\and 
 Jade Master\thanks{Leverhulme Trust Research Project Grant RPG-2020-232} \\ Strathclyde University, United Kingdom \\ jade.master@strath.ac.uk
 \and
 Clemens Kupke$^*$ \\
 Strathclyde University, United Kingdom \\ clemens.kupke@strath.ac.uk}


\runninghead{E.\ Schoen, J.\ Master, and C.\ Kupke}{Relative fixed points of functors} 


\maketitle

\begin{abstract}
We show how the relatively initial and relatively terminal fixed points studied by Ad\'amek et al.~for a well-behaved functor $F$ form a pair of adjoint functors between $\F$-coalgebras and $\F$-algebras. We use the language of locally presentable categories to find sufficient conditions for existence of this adjunction. We show that relative fixed points may be characterized as (co)equalizers of the free (co)monad on $F$. In particular, when $F$ is a polynomial functor on $\Set$ the relative fixed points are a quotient or subset of the free term algebra or the cofree term coalgebra. We give examples of the relative fixed points for polynomial functors and a presentation of the Sierpinski carpet as a relative fixed point. Lastly, we prove a general preservation result for relative fixed points. \end{abstract}

\begin{keywords} 
coalgebra, algebra, fixed points, coalgebra-to-algebra morphisms
\end{keywords} 

\section{Introduction}
Fixed points of functors are particularly relevant to the study of coalgebras. As in \cite{jacobs2017introduction,adamek2018fixed} these fixed points capture the ideas of induction and coinduction on coalgebras. The main focus of research into fixed points of functors has thus far been on either the least fixed point of a functor or the greatest fixed point of a functor. However, in general a functor has more fixed points than just these two. We call these additional fixed points ``relative fixed points", after the ``relatively terminal coalgebras" of \cite{adamek2012relatively} and the ``relatively initial algebras" of \cite{adamek2016howmany}. Other constructions yielding relative fixed points are the rational fixed points of Ad{\'a}mek, Milius, and Velebil \cite{adamek2006iterative} and the locally finite fixed points of Milius, Pattinson, and Wi{\ss}mann \cite{milius2016new}. The main contribution of this paper is a presentation of relative fixed points via a pair of adjoint functors
\begin{equation} \hfill \label{diag:adj}
\begin{tikzcd}
\FCoalg \ar[r,bend left,"\mu"] \ar[r,phantom,"\bot"]& \FAlg \ar[l,bend left,"\nu"]
\end{tikzcd}
\hfill \end{equation}
This adjunction reveals the deep connection between relative fixed points and coalgebra-to-algebra homomorphisms (abreviated as ca-morphism).
Algebras and coalgebras which have unique ca-mor\-phisms going into or out of them have been studied extensively in \cite{adamekLuckeMilius2007,capretta2009corecursive,capretta2006recursive,levy2015final,adamek2020wellfounded} due to their connection to (co)induction principles. However, the fixed points studied in this paper are universal with respect to ca-morphisms which may not be unique. In the context of functional programming \cite{meijer1991functional}, not-necessarily unique ca-morphisms are used as data structures for recursion schemes. In \cite{hauhs2015scientific}, the authors argue for the use of non-unique ca-morphisms as a framework for scientific modelling. In this paper, we study the classifying objects for ca-morphisms in the context of the adjunction they induce. 

The paper is organised as follows:
In Section \ref{sec:universal}, we will introduce relative fixed points for $F$-(co)algebras, and the adjunction~\eqref{diag:adj} in the full categorical case. After that we will provide sufficient conditions for the existence of the adjunction. In Section \ref{sec:constructions}, we will provide examples of these relative fixed points as well as an explicit characterization for polynomial functors. In Section \ref{sec:preservation}, we will discuss when the adjunction is preserved by a functor and give some important examples of this phenomenon. Finally, in Section~\ref{sec:conclusions} we will draw conclusions and point to ideas for future work.

We now finish the introduction with a brief discussion of relative fixed points for monotone functions to equip the reader with some intuitions before moving to the more general categorical setting that follows.

\subsubsection*{Warm-up: Relative fixed points of monotone functions}


Consider a monotone function $f: L \to L$ on a complete lattice $L$. It has a least and a greatest fixed point, which may be constructed by an `approximation process' which generalizes Kleene's fixed point theorem for continuous functions \cite{cousot1979constructive}. For example, let $f$ be the following monotone function on $([0,1],\leq)$ the interval of real numbers with the usual ordering: 
\begin{center}
 \includegraphics[scale=.33]{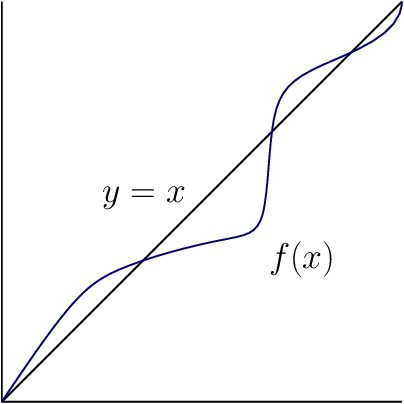}
 \end{center}
The function $f$ is overlayed with the function $y=x$. The intersection of the two curves indicate fixed points of $f$. The least fixed point of $f$ is $0$ and the greatest fixed point is $1$ but there are $3$ other fixed points in-between. These relative fixed points have a similar construction to the least and greatest ones. Given a ``post-fixed point" i.e. a point $x \in [0,1]$ such that $x \leq f(x)$ we may find the first fixed point above $x$ as
\[ \hfill\mu(x) = \sup \{x,f(x),f^2(x),f^3(x),\ldots\} \hfill \]
where the $\ldots$ indicate iteration to a sufficiently large ordinal. Similarly, given a ``pre-fixed point" $f(y) \leq y$, we may find the closest fixed point below $y$ as
\[\hfill\nu(y) = \inf \{ y ,f(y),f^2(y),f^3(y),\ldots\}\hfill\]  For a complete lattice $L$, let $Pre(f)$ be the suborder of $L$ consisting of only the post-fixed points $x \leq f(x)$. Similarly, let $Post(f)$ be the suborder of pre-fixed points $f(y) \leq y$. Then there is a Galois connection
\[ \hfill
\begin{tikzcd}
Post(f) \ar[r,bend left,"\mu_f"] \ar[r,phantom,"\bot"] & Pre(f) \ar[l,bend left,"\nu_f"]
\end{tikzcd}
\hfill \]
Being a Galois connection means that
\[ \hfill\mu(x) \leq y \iff x \leq \nu(y) \hfill \]
In this paper we will generalize this Galois connection to fixed points of functors rather than monotone functions. When generalizing from posets to categories we make the replacements shown in Figure \ref{table}. 

\begin{figure*}[h]
\begin{center}
\begin{tabular}{|c|c|}
\hline
    Poset & Category \\
    \hline
    Monotone Function $f$ & Functor $\F$ \\
    Post-fixed point of $f$ & $\F$-coalgebra \\
    Pre-fixed point of $f$ & $\F$-algebra \\ $\sup \{f(x),f^2(x),f^3(x), \ldots \}$ & $\colim(X \to F(X) \to F^2(X) \to F^3(X) \ldots)$ \\
    $\inf\{f(x),f^2(x),f^3(x), \ldots \}$ &  $\lim (X \leftarrow F(X) \leftarrow F^2(X) \leftarrow F^3(X) \ldots )$ \\
    Galois connection & Adjunction\\
    \hline
\end{tabular}
\end{center}
\vspace{-1ex}
\caption{Generalization from Posets to Categories}\label{table}
\end{figure*}

\paragraph{Acknowledgements.} For  insightful comments and questions, thank you to every member of the Mathematically Structured Programming group, Corina Cirstea, Toby Wilkinson, and Alexandre Goy.

\section{Relative Fixed Points are Adjoint}\label{sec:universal}

In this section, we recall the definition of `relatively terminal coalgebra' from \cite{adamek2012relatively} and the dual notion of `relatively initial algebra' from \cite{adamek2016howmany}. As is usual for definitions via universal properties, there may or may not be an object enjoying the property;  however, if there is one, it is unique up to unique isomorphism. 

\begin{definition}
   For an algebra $a$ and a coalgebra $b$, a coalgebra-to-algebra morphism from $a$ to $b$ (abbreviated as ca-morphism) is a morphism $f \maps B \to A$ making the following diagram commute:
    \[ \hfill
    \begin{tikzcd}
        FB\arrow[r, "Ff"] & FA\arrow[d, "a"]\\
        B\arrow[u, "b"] \arrow[r, "f"] & A
    \end{tikzcd}
    \hfill \]
Let $\Hylo(b,a)$ denote the set of coalgebra-to-algebra morphisms from $b$ to $a$.
\end{definition}

Note that $\Hylo$ constitutes the object part of a functor $\FCoalg^\op\times \FAlg\to \Set$ (i.e. a profunctor $\FCoalg \nrightarrow \FAlg$). Its action on morphisms is given via composition, as seen in the following diagram:
\[
\begin{tikzcd}
FB'\arrow[r, "Ff"] & FB\arrow[r, "Fg"] & FA\arrow[r, "Fh"] \arrow[d, "a"]& FA'\arrow[d, "a'"] \\
B'\arrow[u, "b'"] \arrow[r, "f"] & B\arrow[u, "b"] \arrow[r, "g"] & A\arrow[r, "h"] & A'
\end{tikzcd}
\]
That is, ca-morphisms can be precomposed with coalgebra morphisms, and postcomposed with algebra morphisms. We now define relative fixpoints as (co)representing objects for $\Hylo$:

\begin{definition}\label{universalprop}
Suppose we have an algebra $a : F A \to A$. A coalgebra $b:B\to FB$ is called \emph{terminal relative to $a$} if there is a natural isomorphism
\[
\hfill\phi\maps\FCoalg(-,b)\cong \Hylo(-, a)\hfill
\]
Similarly, for a coalgebra $b': B'\to FB'$, an algebra $a':FA'\to A'$ is called \emph{initial relative to $b'$} if there is a natural isomorphism
\[
\hfill\psi\maps\FAlg(a',-)\cong \Hylo(b', -)\hfill
\]
\end{definition}

By the Yoneda lemma, if an algebra $a$ admits a relatively terminal coalgebra, it must be unique up to unique isomorphism; hence, we may use the functional notation $\nu(a)$ or $\nu a$ to denote \emph{the} coalgebra which is terminal relative to $a$. Similarly, we will write $\mu(b)$ or $\mu b$ for \emph{the} algebra which is initial relative to $b$. 

However, note that so far we have no guarantee as to the existence of $\nu(a)$ and $\mu(b)$; we will adopt the convention that the use of an expression $\nu (a)$ or $\mu (b)$ carries with it the implicit assumption that such an object exists. So for example, Proposition~\ref{prop:fixed} should be read as ``For any algebra $a$, \emph{if a relatively terminal coalgebra exists}, it is a fixed point of $F$". In Theorem~\ref{thm:accessible-existence}, we will show that under appropriate conditions, $\mu$ and $\nu$ define total functors. 

As in the Yoneda lemma, of central importance are the maps $\eta = \psi(\id_{\mu (b)})$ and $\epsilon = \phi(\id_{\nu (a)})$. We have the following lemma:
\begin{lemma}
Let $b:B\to FB$ be an $F$-coalgebra, $a:FA\to A$ an $F$-algebra. Let $f:b\to \nu a$ be a coalgebra morphism, and $g:\mu b\to a$ an algebra morphism. Then we have the equalities
\begin{align}
\phi(f) = \epsilon\circ f\label{eq:phieps}\\
\psi(g) = g\circ\eta\label{eq:psieta}
\end{align}
\end{lemma}
\begin{proof}
We prove \eqref{eq:phieps}, since \eqref{eq:psieta} follows by duality. We simply note that 
\[
\epsilon\circ f = \phi(\id_{\nu a})\circ f = \phi(\id_{\nu a}\circ f) = \phi(f)
\]
by naturality of $\phi$. 
\end{proof}

Perhaps surprisingly, the universal properties of $\mu (b)$ and $\nu (a)$ imply that they are always fixed points for $F$. 
\begin{proposition}\label{prop:fixed}
For any algebra $a$, the coalgebra $\nu (a):\nu A\to F\nu A$ is a fixed point of $F$. Similarly, for any coalgebra $b$, the algebra $\mu (b): F\mu (B)\to \mu(B)$ is a fixed point of $F$. 
\end{proposition}
\begin{proof}This proposition resembles Lambek's lemma; indeed, it is possible to exhibit $\mu (b)$ as an initial algebra for a well-chosen functor $F_b:\namedcat{C}/B\to \namedcat{C}/B$. This is (up to duality) the approach taken in \cite{adamek2012relatively}. For concreteness, We have chosen to give an explicit proof. We prove that $\mu (b)$ is a fixed point; the case for $\nu (a)$ follows by duality. 

We wish to find an inverse $\beta$ to $\mu (b):F(\mu B)\to \mu B$. Since $F(\mu B)$ carries the algebra structure $F(\mu (b)):FF(\mu B)\to F(\mu B)$, it suffices to find a ca-morphism $b\to F(\mu (b))$. This is given by the following diagram:
\[
\begin{tikzcd}
FB\arrow[r, "Fb"] & FFB\arrow[r, "FF\eta"] & FF\mu B\arrow[d, "F\mu (b)"]\\
B\arrow[u, "b"]\arrow[r, "b"] & FB\arrow[r, "F\eta"]\arrow[u, "Fb"] & F\mu B
\end{tikzcd}
\]
This yields an algebra morphism $\beta:\mu B\to F\mu B$ such that 
\begin{equation}
    \beta\eta = \psi(\beta) = (F\eta) b \label{eq:beta}
\end{equation}
It remains to show that $\beta$ is a two-sided inverse to $\mu (b)$. Consider the composite $\mu (b)\circ \beta:\mu B\to \mu B$. We claim that under the correspondence $\psi$, this composite corresponds to $\eta$; since $\psi$ is bijective, and $\id_{\mu B}$ corresponds to $\eta$ by definition, this yields $\mu (b)\circ \beta = \id_{\mu B}$. To verify, we use equality \eqref{eq:psieta}:
\[
\begin{tikzcd}
\mu B\arrow[drr, "\beta", bend left]\\
&FB\arrow[r, "F\eta"]&F\mu B\arrow[d, "\mu (b)"]\\
B\arrow[uu, "\eta"]\arrow[ur, "b"]\arrow[rr, "\eta"]&&\mu B
\end{tikzcd}
\]
The top square is equation \eqref{eq:beta}, and the bottom square commutes since $\eta$ is a ca-morphism. 

We may now conclude that $\mu (b)\circ\beta = \id_{\mu B}$. To show that $\beta\circ\mu (b) = \id_{F\mu B}$, we simply note that $\beta$ is an algebra morphism, and hence
\[
\begin{tikzcd}
F\mu B\arrow[r, "F\beta"]\arrow[d, "\mu (b)"]& FF\mu B\arrow[d, "F\mu (b)"]\\
\mu B\arrow[r, "\beta"]&F\mu B
\end{tikzcd}
\]
commutes. The composite $F\mu (b)\circ F\beta = F(\mu (b)\circ \beta)$ is equal to $F\id_{\mu B} = \id_{F\mu B}$ as already shown, so we also have $\beta\circ\mu (b) = \id_{F\mu B}$. 
\end{proof}

We now give an adjunction characterizing relative fixed points.

\begin{theorem}\label{thm:adj}
Let $\C$ be a category, and $F:\C\to \C$ an endofunctor. Assume that every $F$-algebra has a relatively terminal coalgebra, and every $F$-coalgebra has a relatively initial algebra. Then $\nu:\FAlg\to \FCoalg$ and $\mu:\FCoalg\to \FAlg$ are the object parts of two adjoint functors
\[ \hfill
\begin{tikzcd}
\FCoalg \ar[r,bend left,"\mu"] \ar[r,phantom,"\bot"]& \FAlg \ar[l,bend left,"\nu"]
\end{tikzcd}
\hfill \]
\end{theorem}

\begin{proof}
For the action of $\mu$ on morphisms, consider a coalgebra morphism
\[
\hfill\begin{tikzcd}FB\arrow[r, "Ff"] &FB'\\ B\arrow[r, "f"]\arrow[u, "b"] & B'\arrow[u, "b'"]\end{tikzcd}\hfill
\]
Then we obtain a ca-morphism
\[
\hfill\begin{tikzcd}
FB\arrow[r, "Ff"] & FB'\arrow[r, "F\eta"] & F\mu B'\arrow[d, "\mu b'"]\\
B\arrow[u, "b"] \arrow[r, "f"] & B'\arrow[u, "b'"] \arrow[r, "\eta"] & \mu B'
\end{tikzcd}\hfill
\]
So, we can set $\mu(f)$ to be $\psi^{-1}(\eta\circ f)$. This preserves composition: note that $\mu(f)$ is the unique algebra morphism $g:\mu B\to \mu B'$ such that $g\circ \eta = \eta\circ f$. Now consider the following diagram:
\[
\begin{tikzcd}
B\arrow[d, "\eta"] \arrow[r, "f"] & B'\arrow[d, "\eta"] \arrow[r, "f'"] & B'' \arrow[d, "\eta"]\\
\mu B\arrow[r, "\mu(f)"] & \mu B'\arrow[r, "\mu(f)'"] & \mu B''
\end{tikzcd}
\]
We see that $\mu(f')\circ\mu(f)$ is an algebra morphism satisfying $\mu (f')\circ \mu (f) \circ \eta = \eta\circ (f'\circ f)$, and so by uniqueness we must have
\[
\mu (f'\circ f) = \mu (f')\circ \mu (f)
\]
The action of $\nu$ on morphisms is defined dually.

To see that $\mu \dashv \nu$, simply consider the composite isomorphism
\[
\FAlg(\mu (b), a) \overset{\psi}\cong \Hylo(b,a)\overset{\phi^{-1}}\cong \FCoalg(b, \nu (a))
\]
We'll show that this is natural in $b$ (naturality in $a$ follows similarly). That is, we need to show that if $g:\mu(b)\to a$ is an algebra morphism, and $f:b'\to b$ is a coalgebra morphism, then
\[
\phi^{-1}\psi(g\circ \mu(f)) = \phi^{-1}\psi(g)\circ f
\]
We have the following chain of equalities:
\begin{align*}
\phi^{-1}\psi(g\circ \mu(f)) &= \phi^{-1}(g\circ \mu(f)\circ \eta) & (\text{equation \ref{eq:psieta}})\\
&= \phi^{-1}(g\circ \eta\circ f) &(\text{definition of }\mu(f))\\
&= \phi^{-1}(g\circ \eta)\circ f & (\text{naturality of }\phi)\\
&= \phi^{-1}\psi(g) &(\text{equation \ref{eq:psieta}})
\end{align*}
\end{proof}

In order to apply Theorem \ref{thm:adj}, we need to know when its assumptions are satisfied. As exposited in the warmup section, an approximation strategy may be employed; however, since categories are usually proper classes, we need to impose some smallness conditions on the category and the functor. 
\begin{definition}
Let $\C$ be a category, and $F:\C\to\D$ a functor.
\begin{enumerate}[label = (\roman*)]
\item We call $\C$ \emph{$\lambda$-chain-complete} if it has limits of diagrams of shape $(\alpha, \geq)$ for all ordinals $\alpha \leq \lambda$.
\item We call $\C$ \emph{$\lambda$-chain-cocomplete} if it has colimits of diagrams of shape $(\alpha,\leq)$ for all ordinals $\alpha \leq \lambda$. 
\item We call $F$ \emph{$\lambda$-continuous} if it preserves limits of diagrams of shape $(\lambda, \geq)$.
\item We call $F$ \emph{$\lambda$-cocontinuous} if it preserves colimits of diagrams of shape $(\lambda,\leq)$.
\end{enumerate}
\end{definition}
With this terminology, we have the following proposition, which is a minor variation on existing results:
\begin{proposition}\label{prop:iterfix}
Let $\C$ be a category, $F:\C\to \C$ an endofunctor, and $\lambda$ a regular (infinite) cardinal.
\begin{enumerate}[label = (\roman*)]
\item If $\C$ is $\lambda$-chain-cocomplete, and $F$ is $\lambda$-cocontinuous, then each $F$-coalgebra admits a relatively initial algebra.
\item If $\C$ is $\lambda$-chain-complete, and $F$ is $\lambda$-continuous, then each $F$-algebra admits a relatively terminal coalgebra.
\end{enumerate}
\end{proposition}

\begin{prfsketch}
We only prove (i), since (ii) follows by duality. Recall that cocontinuous endofunctors on $\lambda$-chain-cocomplete categories have initial algebras (Theorem 3.5 in \cite{adamek2018fixed}; note that there it is assumed that $\C$ has colimits of all chains, but clearly only chains of length at most $\lambda$ are needed). 

Now, dual to lemma 3.2 in \cite{adamek2012relatively}, a relatively initial algebra for $b:B\to FB$ is the same as an initial algebra for $F_b:(B\setminus \C)\to (B\setminus \C)$ defined as
\[
F_b:(B\overset{x}{\to}X) \mapsto (B\overset{b}{\to} FB\overset{Fx}{\to} FX)
\]
Since the `codomain functor' $(B\setminus \C)\to \C$ given by $(B\overset{x}\to X)\mapsto X$ creates colimits, we know that $(B\setminus \C)$ has all colimits that $\C$ has, and hence is also $\lambda$-chain-cocomplete. Moreover, $F_b$ preserves all colimits that are preserved by $F$, and hence is also $\lambda$-cocontinuous. We conclude that $F_b$ has an initial algebra, which is a relatively initial $F$-algebra for $b$. 

Unpacking the abstract theorems used yields the construction of $\mu b$ as the colimit of the $\lambda$-chain
\[
\begin{tikzcd}
B\arrow[r, "b"] & FB\arrow[r, "Fb"] & FFB\arrow[r, "FFb"] & \cdots &(< \lambda)
\end{tikzcd}
\]
\end{prfsketch}

Finally, if $\C$ is locally presentable and $F$ is locally accessible, then both $\mu$ and $\nu$ constitute total functors:
\begin{theorem}\label{thm:accessible-existence}
Let $\C$ be a locally ($\lambda$-)presentable category, and $F:\C\to\C$ a locally ($\lambda$-)accessible functor. Then each $F$-coalgebra has a relatively initial algebra, and each $F$-algebra has a relatively terminal coalgebra. 
\end{theorem}

The existence of relatively terminal coalgebras is proved as Corollary 5.1 in \cite{adamek2012relatively}; we present a novel proof using the Special Adjoint Functor Theorem.

\begin{prfsketch}
If $\C$ is locally presentable, it is cocomplete, so in particular $\lambda$-chain-cocomplete. Additionally, if $F$ is $\lambda$-accessible, it is $\lambda$-cocontinuous. So, by Proposition \ref{prop:iterfix}, all $F$-coalgebras have relatively initial algebras. Hence, we obtain a total functor $\mu:\FCoalg \to \FAlg$. 

To show the existence of $\nu$ we use the (dual of) the special adjoint functor theorem (e.g. \cite[Thm.\ 4.58]{riehl2017category}). By \cite[Exercise 2j]{adamek1994locally}, $\FCoalg$ is locally presentable and by \cite[Corr.\ 2.75]{adamek1994locally} so is the category $\FAlg$. By \cite[Thm.\ 1.58]{adamek1994locally}, both these categories are co-wellpowered. The functor $\mu$ preserves colimits, because it is constructed as a colimit and colimits distribute over themselves. Therefore, by the special adjoint functor theorem, $\mu$ has right adjoint $\nu$; to see that that $\nu (a)$ is terminal relative to $a$, consider the natural equivalences
\[
\FCoalg(b, \nu (a)) \cong \FAlg(\mu (b) , a)\cong \Hylo(b,a).
\]
\end{prfsketch}

Note that in the above theorem, while $\mu$ is constructed as a colimit of an initial chain, $\nu$ is not constructed as a limit of a terminal cochain (not even implicitly, as the Adjoint Functor Theorem yields a \emph{colimit} construction rather than a limit); hence, we unfortunately do not settle open problem \cite[Open Problem 2.7]{adamek2012relatively}. 

We end this section with a few miscellaneous results. First, we can get a clearer idea of the monad and comonad induced by the adjunction.

\begin{remark}\label{rem:flip}
It is easy to see that if $b:B\to FB$ is an isomorphism, then $b^{-1}:FB\to B$ is initial relative to $b$; hence, $\mu (b) = b^{-1}$. Similarly, $\nu \alpha = \alpha^{-1}$ whenever $\alpha:FA\to A$ is an isomorphism. From this, it follows that the monad $\nu\mu : \FCoalg\to \FCoalg$ maps $b:B\to FB$ to
\[
(\mu (b))^{-1}:\mu B\to F\mu B
\]
Similarly, the comonad $\mu\nu$ maps $a:FA\to A$ to $(\nu(a))^{-1}:F\nu A\to \nu A$. 
\end{remark}

Next, the presentations of $\mu$ and $\nu$ in terms of (co)limits of (co)chains is analogous to similar constructions of initial algebras and terminal coalgebras. This connection still holds even in the absence of (co)limits:

\begin{remark}
Let $1$ be the terminal object of $\C$ and let $0$ be the initial object. Then there is a unique algebra $1:F1 \to 1$ and $\nu(1)$ is the terminal coalgebra. Similarly, the initial algebra is given by $\mu(0)$ for the unique coalgebra $0: 0 \to F0$. 
\end{remark}

In fact, this is a special case of a more general result on recursive coalgebras and corecursive algebras. 

\begin{definition}
An $\F$-algebra $a$ is corecursive if for every $\F$-coalgebra $b$ there is a unique ca-morphism from $b \to a$. Dually, an $\F$-coalgebra $b$ is recursive if for any $\F$-algebra $a$, there is a unique ca-morphism $b \to a$.
\end{definition} 

Recursivity of a coalgebra relates to the termination of that coalgebra when thought of as a program (c.f. \cite{adamekLuckeMilius2007}). Note that the coalgebra $0\to F0$ is always recursive, and the algebra $F1\to 1$ is always corecursive. The following corollary connects (co)recursivity to the $\mu-\nu$ adjunction:
\begin{corollary}\label{coro:recursive}
A coalgebra $b:B\to FB$ for which $\mu (b)$ exists is recursive if and only if $\mu (b)$ is initial; similarly, an algebra $a:FA\to A$ for which $\nu (a)$ exists is corecursive if and only if $\nu (a)$ is terminal. 
\end{corollary}
\begin{proof}
This can be easily read off: $b$ is recursive if and only if $\Hylo(b, a)$ always has a unique element, and $\mu (b)$ is initial if and only if $\FAlg(\mu (b), a)$ always has a unique element. Since
\[
\Hylo(b,a) \cong \FAlg(\mu (b), a)
\]
by definition, the equivalence follows. The second statement follows analogously.
\end{proof}

\section{Concrete Constructions of Relative Fixed Points}\label{sec:constructions}

In this section we provide several concrete constructions of relative fixed points, using a presentation of $\mu$ and $\nu$ based on (co)free (co)algebras.  In Examples \ref{ex1}, \ref{ex2} we explore relative fixed points of polynomial functors and discuss their interpretations. Next, in Proposition \ref{polytopos}, we construct a downward fixed point which classifies Cartesian subcoalgebras in the sense of \cite{adamek2020wellfounded}. In Example \ref{ex:sierp} we illustrate how the Sierpinski carpet may be constructed as a relatively terminal coalgebra. Lastly, we show in Example \ref{ex0}, how the depleted version of the adjunction, that is the Galois connection between post-fixed points and pre-fixed points mentioned in the introduction, may be useful for the Safety Problem as stated in \cite{kori2023exploiting}. 

\begin{proposition}\label{polyexist}
For a polynomial functor $F: \Set \to \Set$, each coalgebra admits a relatively initial algebra, and every algebra admits a relatively terminal coalgebra. 
\end{proposition}
\begin{proof}
Theorem \ref{thm:accessible-existence} guarantees their existence if $\F$ is accessible. Let $F = \sum_{i\in I}y^{X_i}$, and let $\lambda$ be a regular cardinal, such that $\lambda \geq \sup\{|X_i|\mid i\in I\}$. Then each $y^{X_i}$ is $\lambda$-accessible, and hence so is their coproduct $F$. 
\end{proof}

In order to give explicit descriptions for $\mu$ and $\nu$ on $\Set$, we exploit the fact that free algebras and cofree coalgebras for polynomial functors on $\Set$ have elegant characterizations. For completeness' sake, we recall the definition of (co)free (co)algebras.
\begin{definition}
Let $F:\C\to\C$ be a functor, and $X$ an object of $\C$. \begin{itemize}
\item A \emph{free algebra on $X$} is an $F$-algebra $\alpha:FT_X\to T_X$ together with a map $\inc:X\to T_X$ such that for each algebra $a:FA\to A$ equipped with a map $j:X\to A$, there is a unique algebra morphism $\ulcorner j\urcorner:T_X\to A$ such that $j = \ulcorner j\urcorner \circ \inc$; we'll call the morphism $\ulcorner j\urcorner$ the \emph{extension} of $j$. 
\item A \emph{cofree coalgebra on $X$} is an $F$-coalgebra $\gamma:C_X\to FC_X$ together with a map $\outc:C_X\to X$ such that for each coalgebra $b:B\to FB$ equipped with a map $c:B\to X$, there is a unique coalgebra morphism $\llcorner c\lrcorner : B\to C_X$ such that $c = \outc\circ \llcorner c\lrcorner$; we'll call the morphism $\llcorner c\lrcorner$ the \emph{coextension} of $c$. 
\end{itemize}
\end{definition}
In the descriptions of (co)free (co)algebras of polynomial functors, we use the notion of a \emph{$\Sigma$-branching tree}. A $\Sigma$-branching tree is a tree $t$ such that each node $u$ in $t$ comes equipped with a choice of $\sigma_u\in \Sigma$ and has children $v_1,\dots, v_{\Ar(\sigma_u)}$, as depicted in Figure \ref{fig:sigmatree}. 

\begin{figure}
\centering
\begin{tikzpicture}[shorten <=10pt, shorten >=10pt]
\node at (0,0) {$u$};
\node at (-1,1.2) {$v_1$};
\node at (1,1.2) {$v_k$};
\node at (0,1.2) {$\dots$};
\draw (0,0) -- (-1,1.2);
\draw (0,0) -- (1,1.2);
\node at (0,0.7) {$\sigma$};
\end{tikzpicture}
\caption{\label{fig:sigmatree}}
\end{figure}
\begin{proposition}
Let $\Sigma$ be a set of symbols, and for each $\sigma\in \Sigma$, fix an arity $\Ar(\sigma)\in\N$. Let $F:\Set\to\Set$ be the polynomial functor given by $FX = \sum_{\sigma\in\Sigma}X^{\Ar(\sigma)}$. Then,
\begin{enumerate}[label = (\roman*)]
\item the free $\F$-algebra on $X$, denoted $T^\Sigma(X)$, is given by the set of finite $\Sigma$-branching trees with leaves labeled by elements of $X$. Equivalently, $T^\Sigma(X)$ is the algebra of $\Sigma$-terms over $X$, known from universal algebra (see \cite{burris1981univalg} for further description of free algebras, as well as congruences and quotients of $F$-algebras). 
\item The cofree $\F$-coalgebra on $X$, denoted $C^\Sigma(X)$, is given by the set of finite and infinite $\Sigma$-branching trees with internal nodes labeled by elements of $X$. 
\end{enumerate}
\end{proposition}
The following constructions allows us to describe $\mu$ and $\nu$ in terms of (co)free (co)algebras:
\begin{theorem}
Let $\C$ be a category, and $F:\C\to \C$ an endofunctor. Let $b:B\to FB$ be a coalgebra, and $a:FA\to A$ an algebra.
\begin{itemize}
    \item Assume that $B$ admits a free algebra $T^FB$, with unit $\inc:B\to T^FB$ and free algebra structure $\alpha : FT^FB\to T^FB$. Then a relatively initial algebra for $b$ is the same as a coequalizer of the diagram
    \[
    \begin{tikzcd}
    T^{F}(B) \ar[r,shift left=0.5ex,"id"] \ar[r,shift right=0.5ex,"\unfold",swap] & T^{F}(B)
    \end{tikzcd}\]
    in the category of $F$-algebras and where $\unfold$ is the free extension of the following map to $T^F(B)$
    \[B \xrightarrow{b} FB \xrightarrow{F \inc} F T^F (B) \xrightarrow{\alpha} T^F (B)\]
    \item Assume that $A$ admits a cofree coalgebra $C^FA$, with counit $\outc:C^FA\to A$ and cofree coalgebra structure $\gamma : C^FA\to FC^FA$. Then a relatively terminal coalgebra for $a$ is the same as an equalizer of the diagram
    \[\begin{tikzcd}
        C^{F}(A) \ar[r,shift left=0.5ex,"id"] \ar[r,shift right=0.5ex,"\mathsf{pred}",swap] & C^{F}(A) 
        \end{tikzcd}
        \]
        in the category of $F$-coalgebras where $\mathsf{pred}$ is the coextension of the following map to $C^F(A)$
        \[C^{F}(A) \xrightarrow{\gamma} F C^F (A) \xrightarrow{F \outc} FA \xrightarrow{a} A\]
\end{itemize}
\end{theorem}
\begin{proof}
We only prove the statement for $\mu$, since the statement for $\nu$ follows by duality. Our strategy will be to exhibit two isomorphic categories $\Coeq(\id, \unfold)$ and $\int \Hylo(b, -)$ such that an initial object of $\Coeq(\id,\unfold)$ is a coequalizer of $\id$ and $\unfold$, and an initial object of $\int \Hylo(b,-)$ is a relatively initial algebra for $b$. 

First, consider the category $\mathsf{E}$ with objects algebras $a:FA\to A$ together with a map $j:B\to A$. $T^F(B)$ being a free $F$-algebra on $B$ means that for each $(j,a)$ in $\mathsf{E}$, there is a unique algebra morphism $\ulcorner j \urcorner:T^F(B)\to A$ such that $j = \ulcorner j\urcorner \circ \inc$. It is easy to check that this extends to a functor $\ulcorner-\urcorner:\mathsf{E}\to T^F(B)\setminus \FAlg$, where $T^F(B)\setminus \FAlg$ denotes the coslice category over $T^F(B)$. On the other hand, there is a functor $-\circ \inc : T^F(B)\setminus \FAlg \to \mathsf{E}$ given by
\[
(f:T^F(B)\to A) \mapsto (f\circ \inc : B\to A)
\]
In fact, $-\circ \inc$ and $\ulcorner-\urcorner$ are two-sided inverses, since $\ulcorner j\urcorner\circ \inc = j$ by definition, and $\ulcorner f\circ \inc\urcorner$ and $f$ are both algebra morphisms $g$ satisfying $g\circ \inc = f\circ \inc$, so must be the same by uniqueness of $\ulcorner-\urcorner$. 

So $\mathsf{E}$ and $T^F(B)\setminus \FAlg$ are isomorphic. Now let $\int\Hylo(b,-)$ be the full subcategory of $\mathsf{E}$ consisting of those $(j:B\to A)$ which are ca-morphisms $b\to a$, and let $\Coeq(\id, \unfold)$ be the full subcategory of $T^F(B)\setminus \FAlg$ consisting of those $f:T^F(B)\to A$ such that $f = f\circ\unfold$. We claim that the isomorphism $\mathsf{E}\cong (T^F(B)\setminus \FAlg)$ restricts to an isomorphism $\int \Hylo(b,-)\cong \Coeq(\id,\unfold)$. 

So, we need to show that (i) if $j:b\to a$ is a ca-morphism, then $\ulcorner j\urcorner$ coequalizes $\id$ and $\unfold$, and (ii) if $f:T^F(B)\to A$ coequalizes $\id$ and $\unfold$, then $f\circ \inc$ is a ca-morphism. 
\begin{enumerate}[label = (\roman*)]
\item Let $j:b\to a$ be a ca-morphism. Note that both $\ulcorner j\urcorner$ and $\ulcorner j\urcorner\circ \unfold$ are algebra morphisms $T^F(B)\to A$; by freeness of $T^F(B)$, to show that they are equal it suffices to show that $\ulcorner j\urcorner\circ \inc = \ulcorner j\urcorner\circ \unfold \circ \inc$. Consider the diagram
    \[\begin{tikzcd}[ampersand replacement=\&]
	\& B \&\&\& FB \\
	\\
	{}\&{T^F(B)} \& {T^F(B)} \&\& {FT^F(B)} \\
	\\
	\& A \&\&\& FA
	\arrow["{\ulcorner j\urcorner}", from=3-3, to=5-2]
	\arrow["{\ulcorner j\urcorner}"{description}, from=3-2, to=5-2]
	\arrow["\inc"{description}, from=1-2, to=3-2]
	\arrow["\unfold\circ \inc", from=1-2, to=3-3]
	\arrow["j", swap, from=1-2, to=5-2, bend right = 50]
	\arrow["a", from=5-5, to=5-2]
	\arrow["{F \ulcorner j\urcorner}"{description}, from=3-5, to=5-5]
	\arrow["{F \inc}"{description}, from=1-5, to=3-5]
	\arrow["b", from=1-2, to=1-5]
	\arrow["{\alpha}", from=3-5, to=3-3]
    \arrow["Fj", from=1-5, to=5-5, bend left = 50]
    \arrow[phantom, from=3-2, to=3-3, "(*)", description]
\end{tikzcd}\]
Here, the facet labeled (*) is the equality to be established. The outer facet commutes since $j$ is assumed to be a ca-morphism. The top right square commutes by definition of $\unfold$. The bottom right square commutes since $\ulcorner j \urcorner$ is an algebra morphism. Since all but one facet commutes, the remaining facet commutes as well, and hence $\ulcorner j\urcorner \circ \inc = \ulcorner j \urcorner \circ \unfold \circ \inc$, from which it follows that $ \ulcorner j\urcorner$ coequalizes $\id$ and $\unfold$.
\item Assume $f:T^F(B)\to A$ is an algebra morphism that coequalizes $\id$ and $\unfold$. Then consider the diagram
\[
\begin{tikzcd}
FB\arrow[r, "F\inc"] & FT^\F(B)\arrow[r, "Ff"] \arrow[d, "\alpha"] & FA\arrow[dd, "a"] \\
 & T^F(B)\arrow[dr, "f"] & \\
B\arrow[uu, "b"] \arrow[r, "\inc", swap]&T^F(B)\arrow[u, "\unfold"{description}]\arrow[r, "f", swap]&A
\end{tikzcd}
\]
Here the left square commutes by definition of $\unfold$. The top right commutes since $f$ is an algebra morphism, and the bottom right commutes since $f = f\circ\unfold$ by assumption.
\end{enumerate}
So, we have seen that indeed, $\Coeq(\id, \unfold)\cong \int \Hylo(b,-)$. Note moreover that an initial object of $\Coeq(\id, \unfold)$ is exactly a coequalizer of $\id$ and $\unfold$, and an initial object of $\int \Hylo(b,-)$ is exactly a relatively initial algebra for $b$. We conclude that (1) an algebra $a:FA\to A$ equipped with a ca-morphism $j:b\to a$ is relatively initial for $b$ if and only if $\ulcorner j \urcorner$ is a coequalizer for $\id$ and $\unfold$, and (2) an algebra morphism $f:T^F(B)\to A$ is a coequalizer for $\id$ and $\unfold$ if and only if $A$ is relatively initial for $b$ with respect to the ca-morphism $f\circ \inc$. 
\end{proof}

Unpacking the above equalizers and coequalizers in the case of polynomial functors on $\Set$ gives the following corollary. The proof of this corollary is left to the reader.
\begin{proposition}\label{prop:polychar}
Let $F:\Set\to\Set$ be a polynomial functor, say $FX = \sum_{\sigma\in\Sigma}X^{\Ar(\sigma)}$.
\begin{enumerate}[label = (\roman*)]
\item Let $b:B\to FB$ be a coalgebra. We can consider both $B$ and $FB$ as subsets of $T^\Sigma(B)$, via $\inc:B\to T^\Sigma(B)$ and $\alpha\circ F\inc:FB\to T^\Sigma(B)$. Now let $\approx_b$ be the congruence on $T^\Sigma(B)$ generated by $\{x\approx b(x) \mid x\in B\}$; then $\mu(b)$ is given by \[T^\Sigma(B)/{\approx_b}.\] 
\item Let $a:FA\to A$ be an $F$-algebra. Let $t\in C^\Sigma(A)$ be an $A$-labeled $\Sigma$-tree, and let $u$ be a node in $t$, in configuration \begin{tikzpicture}[shorten <=10pt, shorten >=10pt, baseline={([yshift=20pt]current bounding box.south)}] \node at (0,0) {$u:x$};\node at (-1,1.2) {$v_1:y_1$};\node at (1,1.2) {$v_k:y_k$};\node at (0,1.2) {$\dots$};\draw (0,0) -- (-1,1.2);\draw (0,0) -- (1,1.2);\node at (0,0.7) {$\sigma$};\end{tikzpicture}. We call $u$ \emph{$a$-guided} if $a(\sigma(y_1,\dots, y_k)) = x$. Then $\nu(a)$ is given by 
\[
\{t\in C^\Sigma(A)\mid \forall u\in t: u\text{ is }a\text{-guided}\}
\]
\end{enumerate}
\end{proposition}

This proposition also shows a connection between the $\nu$-construction, and coequations. To illustrate this, consider what happens in the $\mu$-construction: A coalgebra $b:B\to FB$ is treated as a `(flatly) recursive set of equations' $x \approx b(x)$. Then this set of equations can be used construct a quotient $\mu(b)$ of the free $\F$-algebra. Comparing this to the coalgebra-to-algebra picture, it has been noted before that giving a coalgebra-to-algebra morphism $b\to a$ is akin to solving the system of equations presented by $b$ in the algebra $a$~\cite{adamek2006iterative}. We propose that there is a dual perspective: rather than solving the system of equations $b$ in $a$, one could also see a coalgebra-to-algebra morphism as \emph{solving the coequation $a$ in $b$}. To our knowledge, the `coequations-as-algebras' perspective is new. We can leverage $\nu$ to fit it into the wider spectrum of coequational logic. As demonstrated in \cite{dahlqvist2021coeq}, the most general definition of a coequation is `a subcoalgebra of a cofree coalgebra'. Point (ii) of Proposition \ref{prop:polychar} then shows how each algebra gives rise to a canonical coequation. 

As a final note, we should highlight an important difference between our current approach to (co)equations, and the one common in universal (co)algebras: in the latter, the main notion is that of \emph{satisfaction} of (co)equations, whereas we focus on \emph{solving} (co)equations. A coequation $E\subseteq C^\Sigma(X)$ is satisfied by $b:B\to FB$ if every coalgebra-to-algebra morphism $B\to C^\Sigma(X)$ factors through $E$. It can quickly be seen that for coequations of the form $\nu(a)$, a coalgebra $b:B\to FB$ satisfies $\nu(a)$ if and only if \emph{every} map $B\to A$ is a coalgebra-to-algebra morphism. Such a situation is exceedingly rare. This also shows that only particular coequations can be described as $\nu(a)$. 

We will now use the above theorems to study some explicit examples.
\begin{example}\label{ex1}
Let $F: \Set \to \Set$ be the functor given by $FX = \{\times,\checkmark\} \times X$. Let $a$ be the algebra $a \maps F(X) \to X $ with carrier $X=\{0,1\}$ given by 
\[ \hfill
    (\checkmark,  s)   \mapsto  1-s \text{ and }(\times,  s)  \mapsto  s
\hfill \]
where $s$ is either $0$ or $1$. The algebra may be depicted as 
\[ \hfill\begin{tikzcd} 
0 \ar[loop left]{l}{\times} \ar[r,"\checkmark"]  & \ar[l] 1 \ar[loop right]{r}{\times}
\end{tikzcd} \hfill \]
Then $\nu(a)$ has a carrier given by
\[  \left\{\!{u_1\choose s_1}\!{u_2\choose s_2}\!{u_3\choose s_3}\cdots \in (\{\times,\checkmark\}\times X)^{\omega}\ \middle|\ u_i = \times \!\implies\! s_i=s_{i+1} \text{ and } u_i= \checkmark \!\implies\! s_i=1-s_{i+1}\right\}  \]
i.e. the subset of streams in $(\{\times,\checkmark\})^{\omega}$ which follow the action of $a$ \emph{when read from right to left}. Given a coalgebra $b : B \to \{\times,\checkmark\} \times B$, a coalgebra-to-algebra morphism may represent a solution to the constraint represented by $a$. That is, we divide the states of $B$ into two classes, such that the division `respects the algebra structure on $A$'. If $m$ is such a marking, we obtain
\[ \hfill
\begin{tikzcd}
   \{\times,\checkmark\} \times B \ar[r,"F{\hat m}"] & \{\times, \checkmark\} \times \nu(a)\\
   \ar[r,"\hat m "] B \ar[u, "b"] & \nu(a) \ar[u] 
\end{tikzcd}
\hfill \]
via the universal property of $\nu$. Intuitively, $\hat m$ maps a state $x$ to the stream of `tags and classes' that are observed when running $b$ forwards. The constraint on $m$ then states that whenever a $\checkmark$ is observed, the class must change, whereas whenever a $\times$ is observed, the class must stay the same. A marking satisfying this constraint exists, if and only if on each cycle in $B$, the number of $\checkmark$'s is even. 

\end{example}
\begin{example}\label{ex2}
Consider the coalgebra $\mathcal{A}$ for the functor $FX = \{ \times, \checkmark\} \times X^{\{a,b\}}$ as depicted in Figure~\ref{fig:automaton}
\begin{figure}
\centering
\begin{tikzpicture}[scale = 1.5]
\node at (0,0) {$q_0$};
\node at (-1,-1) {$q_1$};
\node[draw,circle] at (1,-1) {$q_2$};

\draw[<->, shorten <= 7pt, shorten >= 7pt] (0,0) -- node[anchor = south east] {a} (-1,-1);
\draw[->, shorten <= 7pt, shorten >= 12pt] (0,0) -- node[anchor = south west] {b} (1,-1);
\draw[->, shorten <= 7pt, shorten >=7pt] (-1,-1) .. controls (-2,0) and (-2,-2) .. node[anchor = east] {b} (-1,-1);
\draw[->, shorten <=12pt, shorten >=12pt] (1,-1) .. controls (2,0) and (2,-2) .. node[anchor = west] {a,b} (1,-1);
\end{tikzpicture}
\vspace{-4ex}
\caption{\label{fig:automaton}}
\end{figure}
with carrier given by $X=\{q_0,q_1,q_2\}$. Let $\approx$ be the smallest congruence satisfying
\begin{center}
\begin{tikzpicture}
\node at (-0.5,0) {$q_0$};

\node at (2,-0.6) {$\times$};
\node at (1,0.6) {$q_1$};
\node at (3,0.6) {$q_2$};

\node at (0.3,0) {$\approx$};

\draw[shorten <= 10pt, shorten >= 10pt] (2,-0.6) -- node[anchor = north east] {a} (1,0.6);
\draw[shorten <= 10pt, shorten >= 10pt] (2,-0.6) -- node[anchor = north west] {b} (3,0.6);

\begin{scope}[xshift = 5cm]
\node at (-0.5,0) {$q_1$};

\node at (2,-0.6) {$\times$};
\node at (1,0.6) {$q_0$};
\node at (3,0.6) {$q_1$};

\node at (0.3,0) {$\approx$};

\draw[shorten <= 10pt, shorten >= 10pt] (2,-0.6) -- node[anchor = north east] {a} (1,0.6);
\draw[shorten <= 10pt, shorten >= 10pt] (2,-0.6) --  node[anchor = north west] {b} (3,0.6);
\end{scope}

\begin{scope}[xshift = 10cm]
\node at (-0.5,0) {$q_2$};

\node at (2,-0.6) {$\checkmark$};
\node at (1,0.6) {$q_2$};
\node at (3,0.6) {$q_2$};

\node at (0.3,0) {$\approx$};

\draw[shorten <= 10pt, shorten >= 10pt] (2,-0.6) -- node[anchor = north east] {a} (1,0.6);
\draw[shorten <= 10pt, shorten >= 10pt] (2,-0.6) --  node[anchor = north west] {b} (3,0.6);
\end{scope}

\node at (3.75,-0.1) {{,}};
\node at (8.75,-0.1) {{,}};
\end{tikzpicture}
\end{center}
Then the carrier of $\mu(\mathcal{A})$ is given by
\[{\text{ finite }\{a,b\}\text{ branching trees with $\{\times,\checkmark\}$ labeling internal nodes and $X$ labeling leaves}}/{\approx}\]
where the quotient denotes a quotient in $\Set$, i.e. the set in the numerator modulo the equivalence relation $\approx$. One may see it as terms over the 2 binary operations $\times$ and $\checkmark$ in the three unknowns $\{q_0,q_1,q_2\}$, where $q_0,q_1,q_2$ satisfy a mutual recursive relationship. 
\end{example}
\begin{example}\label{ex:polytopos}
Let $\Sigma$ be a set of symbols, and for each $\sigma\in \Sigma$, fix an arity $\Ar(\sigma) \in \N$. Let $F:\Set\to\Set$ be the polynomial functor given by $FX = \sum_{\sigma\in \Sigma}X^{\Ar(\sigma)}$. Since polynomial functors preserve pullbacks, it follows from Corollary 3.2 in \cite{johnstone2001topos} that $\FCoalg$ is an (elementary) topos. Its subobject classifier $\Omega$ is the coalgebra of `non-decreasing $\Sigma$-trees'; that is, the points of $\Omega$ are $\mathbf{2}$-labeled $\Sigma$-trees, where the label of a child may not be smaller than the label of its parent. 

$\Omega$ is not a fixed point; however, there is a subcoalgebra $\Omega_{\str}$ which is a fixed point, and arises as $\nu$ of a well-chosen algebra. Consider the algebra $\bigwedge:F\mathbf{2}\to\mathbf{2}$, explicitly
\[
\hfill
\bigwedge:\sigma(x_1,\dots, x_n) \mapsto \begin{cases}1 & x_i = 1\text{ for all }i = 1,\dots, n\\0&\text{ otherwise}\end{cases}
\hfill
\]
Then $\nu(\bigwedge)$ is a subcoalgebra of $\Omega$; it consists of those non-decreasing $\Sigma$-trees where zeroes `cannot disappear', i.e. if a node is labeled with $0$, at least one of its children is labeled with $0$. \end{example}

$\Omega_{\str}$ satisfies a universal property similar to the subobject classifier in $\FCoalg$; but instead of classifying \emph{all} subcoalgebras, it classifies only the \emph{Cartesian} subcoalgebras i.e., those subcoalgebras $s:S\leq X$ such that the square
\[\hfill
\begin{tikzcd}
S\arrow[r, "s", tail]\arrow[d] & X\arrow[d]\\
FS\arrow[r, "Fs", tail] & FX
\end{tikzcd}\hfill
\]
is a pullback square. Explicitly, that means that there is a map $\top:Z\to \Omega_\str$ (with $Z$ the terminal coalgebra), such that for each coalgebra $X$ and each Cartesian subobject $S\leq X$, there is a unique map $\phi^\str_S:X\to \Omega_\str$ such that
\[
\hfill
\begin{tikzcd}
S\arrow[r]\arrow[d, tail] 
\arrow[dr, phantom, very near start, " "{pullback}]& Z\arrow[d, "\top"]\\
X\arrow[r, "s"] & \Omega_\str
\end{tikzcd}
\hfill
\]
is a pullback square. Ordinary subobjects are understood as `forward stable subsets': they are subsets $S$ such that if $s\in S$, then so are all the successors of $s$. Cartesian subcoalgebras are those subsets which also satisfy the converse implication: if all successors of $s$ are in $S$, then so is $s$. 

\eject

More formally, let $\xi:X\to FX$ be a coalgebra, and consider the `next-time modality' $\nexttime:P(X)\to P(X)$ from \cite{jacobs2017introduction}, defined on a subobject $U\leq X$ via the pullback
\[
\hfill
\begin{tikzcd}
\nexttime(U)\arrow[r, tail]\arrow[d]\arrow[dr, phantom, " "{pullback}, very near start]& X\arrow[d, "\xi"]\\
FU\arrow[r, tail] & FX
\end{tikzcd}
\hfill
\]
Then subcoalgebras are subsets $P\subseteq X$ such that $P\subseteq \nexttime P$; these are classified by $\Omega$. In \cite{adamek2020wellfounded}, they show that Cartesian subcoalgebras are fixed points for $\nexttime$, i.e.\ they satisfy $P = \nexttime P$. 
\begin{proposition}\label{polytopos}
$\Omega_\str$ classifies Cartesian subcoalgebras.
\end{proposition}

\begin{proof}
We wish to show that $\Omega_{\str} = \nu(\bigwedge)$ classifies Cartesian subobjects. We first prove that Cartesian subobjects are closed under pullbacks. 

Assume $P\leq X$ is a Cartesian subcoalgebra. Let $y:Y\to X$ be a coalgebra morphism. Then consider the following cube, where we write $y^*P$ for the pullback of $P$ along $y$:
\[\hfill\begin{tikzcd}
	Fy^*P && FY \\
	& FP && FX \\
	y^*P && Y \\
	& P && X
	\arrow[from=1-1, to=1-3]
	\arrow[from=3-1, to=1-1]
	\arrow[from=3-3, to=1-3]
	\arrow[from=3-1, to=3-3]
	\arrow[from=1-1, to=2-2]
	\arrow[from=1-3, to=2-4]
	\arrow[from=2-2, to=2-4, crossing over]
	\arrow[from=4-4, to=2-4]
	\arrow[from=4-2, to=2-2, crossing over]
	\arrow[from=4-2, to=4-4]
	\arrow[from=3-3, to=4-4]
	\arrow[from=3-1, to=4-2]
\end{tikzcd}\hfill\]
Note that since $\F$ preserves pullbacks, we obtain a unique arrow $y^*P\to Fy^*P$. In the above cube, the front square is a pullback since $P$ is Cartesian, and the bottom square is a pullback by definition of $y^*P$. Hence, we see that taking the top and back square together, as in
\[\hfill
\begin{tikzcd}
y^*P\arrow[r]\arrow[d] & Y\arrow[d]\\
Fy^*P\arrow[r]\arrow[d] & FY\arrow[d]\\
FP\arrow[r] & FX
\end{tikzcd}\hfill
\]
the outer square is a pullback. The bottom square is also a pullback, since $\F$ preserves pullbacks; hence the top square is a pullback, which shows that $y^*P$ is Cartesian. 

\vspace{10pt}

Now consider the terminal object $Z$ in $\FCoalg$; this is the coalgebra of finite and infinite $\Sigma$-branching trees. We note that $\top:Z\to \Omega$, which maps a tree $t$ to $t$ constantly labeled with $1$, factors through $\Omega_\str$; and moreover $\top$ is a \emph{Cartesian} subcoalgebra of $\Omega_\str$. So whenever $P\leq X$ is a pullback of $\top:Z\to\Omega_\str$, $P$ is a Cartesian subcoalgebra. Uniqueness of classifiers $X\to \Omega_\str$ follows from uniqueness of classifiers $X\to\Omega$, so it suffices to show that if $P\leq X$ is Cartesian, there exists a classifier $\phi^\str_P:X\to \Omega_\str$, such that $P = (\phi^\str_P)^*\top$. 

We know that $\FCoalg(X, \Omega_\str)\cong \Hylo(X, \bigwedge)$, so we may equivalently provide a coalgebra-to-algebra map $X\to 2$. We claim that the characteristic function 
\[
\hfill
\chi_P:x\mapsto \begin{cases}1 & x\in P\\0 & \text{ otherwise}\end{cases}
\hfill
\]
is a coalgebra-to-algebra morphism. For, consider an arbitrary $x\in X$. Let $\xi(x) = \sigma(x_1,\dots, x_n)$. We consider two cases.
\begin{enumerate}[label = (\roman*)]
\item If $x\in P$, then since $P$ is a subcoalgebra, we know $x_i\in P$ for all $i$; hence, \[\hfill\bigwedge(\sigma(\chi_P(x_1),\dots, \chi_P(x_n))) = \bigwedge(\sigma(1,\dots, 1)) = 1 = \chi_P(x).\hfill\]
\item If $x\notin P$, then it suffices to show that at least one of the $x_i$ is also not in $P$. Assume towards a contradiction that $x_i\in P$ for all $i$. Then the following square commutes:
\[
\hfill
\begin{tikzcd}
\{*\}\arrow[r, "{*\mapsto x}"]\arrow[d, "{*\mapsto \sigma(x_1,\dots, x_n)}", swap]&X\arrow[d, "\xi"]\\
FP\arrow[r] & FX
\end{tikzcd}
\hfill
\]
hence since $P$ was Cartesian, we conclude that the map $*\mapsto x$ factors through the inclusion $P\rightarrowtail X$. But this amounts to saying $x\in P$, which is not the case. 

We conclude that there is an $x_i$ with $\chi_P(x_i) = 0$, and hence
\[
\hfill \bigwedge(\sigma(\chi_P(x_1),\dots, \chi_P(x_n))) = 0 = \chi_P(x).\hfill
\]
\end{enumerate}
So in both cases, we have $\bigwedge(F\chi_P(\xi(x))) = \chi_P(x)$, which shows that $\chi_P$ is a coalgebra-to-algebra morphism. 

We conclude that there is a unique coalgebra morphism $\phi^\str_P:X\to \Omega_\str$ such that $\chi_P = h\circ \phi^\str_P$, where $h:\Omega_\str\to \mathbf{2}$ is the universal coalgebra-to-algebra morphism, mapping a labeled $\Sigma$-tree to the label of its root. We still need to show that $P$ is the pullback of $\top$ along $\phi^\str_P$. Note, however, that
\[
\hfill
\begin{tikzcd}
Z\arrow[r]\arrow[d, "\top"] & 1\arrow[d, "{*\mapsto 1}"]\\
\Omega_\str\arrow[r, "h"] & \mathbf{2}
\end{tikzcd}
\hfill
\]
is a pullback square (in $\Set$), since if the root node of a non-decreasing $\Sigma$-tree $t$ is labeled by 1, then so are all the other nodes in $t$, and hence $t$ is in the image of $\top$. Hence, we can fill in the following diagram:
\[\hfill
\begin{tikzcd}
P\arrow[d, tail]\arrow[r] & Z\arrow[d, "\top"] \arrow[r] & 1\arrow[d, "{*\mapsto 1}"]\\
X\arrow[r, "\phi^\str_P"]\arrow[rr, "\chi_P", bend right] & \Omega_\str \arrow[r, "h"] & \mathbf{2}
\end{tikzcd}
\hfill\]
Here, the outer square is a pullback, since $\chi_P$ classifies $P$ in $\Set$, and we have just shown that the right-hand side is a pullback as well. Therefore, the left-hand square is a pullback, which finishes the proof. 
\end{proof}

\begin{example}\label{ex:sierp}

In \cite{noquez2021sierpinski}, the Sierpinski carpet is presented as a final coalgebra in a category of `square metric spaces'. In this section we recall this work and then show how the downward fixed point construction $\nu$ gives a more direct way of constructing the Sierpinski carpet as a final coalgebra.
\begin{definition}
Let $\blacksquare$ denote the set $[0,1]^2$ where $[0,1]$ is the real unit interval. Let $\square$ denote the boundary of $\blacksquare$ or explicitly
\[
\hfill
\square = \{ (i,r) : i \in \{0,1\}, r \in [0,1] \} \cup \{ (r,i) : r \in [0,1], i \in \{0,1\} \}
\hfill
\] Let $\MS$ be the category whose objects are metric spaces with diameter less than $2$ and whose morphisms are short maps $f : (X,d) \to (X',d')$ i.e.\ a function $f : X \to X'$ such that $d(x,y) \leq d'(f(x),f(y))$. 
\end{definition}
We are interested in two different metrics on $\square$:
\begin{itemize}
\item The path metric $d_p : \square \times \square \to \R$ with $d_p(x,y)$ given by the length of the shortest path in $\square$ between $x$ and $y$.
\item The taxicab metric $d_t : \square \times \square \to \R$ given by $d_t((s,r),(s',r'))=|s'-s| + |r'-r|$.
\end{itemize}
\begin{definition}
A square metric space is a metric space $(X,d)$ equipped with an injective function $S : \square \hookrightarrow X$ such that for all $x,y \in \square$,
\[
\hfill
d_t(x,y) \leq d(S(x),S(y)) \leq d_p(x,y)
\hfill
\]
A morphism of square metric spaces $f : (X,S) \to (X',S')$ is a short map $f : X \to X'$ such that $S' = f \circ S$. This defines a category $\SqMS$ of square metric spaces and their morphisms.
\end{definition}
A key example of a square metric space is $\square$ itself, equipped with the path metric, and with $S_\square:\square\to\square$ being the identity; this forms the initial object of $\SqMS$. Another important square metric space is $\blacksquare$ with the Euclidean metric, and $S_\blacksquare:\square\to\blacksquare$ the inclusion. 

We now define an endofunctor on square metric spaces for which the Sierpinski carpet is a fixed point. We present the following definitions informally. The full definitions may be found in \cite{noquez2021sierpinski}.
\begin{definition}
Let $M$ be the set $\{0,1,2\}^2 \setminus (1,1)$. For a square metric space $S : \square \to X$, $M \otimes X$ is eight copies of $X$ in a three-by-three grid with the center removed. Mathematically, $M\otimes X$ is the Cartesian product $M \times X$ modulo the smallest equivalence relation which identifies the upper edge of square $(i,j)$ with the lower edge of square $(i,j+1)$ whenever both are valid indices in $M$, and similarly for the left and right edge of adjacent copies. We write $m \otimes x$ to denote the equivalence class of $(m,x)$ in $M \otimes X$. We can equip $M \otimes X$ with the structure of a square metric space. To do this, we first define a metric on $M\times X$ by
\[
d_{M\times X}(m\otimes x, n\otimes y) = \begin{cases} \frac{1}{3}d(x,y) & m = n \\ 2&\text{otherwise}\end{cases}
\]
Then we equip $M\otimes X$ with the coursest metric such that the quotient map $M\times X\to M\otimes X$ is short (see \cite{noquez2021sierpinski} for a detailed description). There is a map $\square \to M \otimes X$ which maps $\square$ injectively to the `outer boundary'. For a short map $f : X \to Y$, there is a short map $M \otimes f : M \otimes X \to M \otimes Y$ given by
$m \otimes x \mapsto m \otimes f(x)$. This defines a functor
\[\hfill M \otimes - \maps \SqMS \to \SqMS \hfill \]
\end{definition}
As shown in \cite{noquez2021sierpinski}, $M \otimes -$ has an initial algebra. $\square$ is an initial object in $\SqMS$ so the initial algebra may be found by taking the colimit of the usual chain 
\[ \hfill
\square \to M \otimes \square \to M \otimes M \otimes \square \to \ldots \hfill
\]
As $\SqMS$ does not have a final object, we cannot construct a final coalgebra by taking the limit of the dual of this chain. However, our construction $\nu$ does not require a final object in the base category.
\begin{figure}[h]
\centering
\includegraphics[scale=.3]{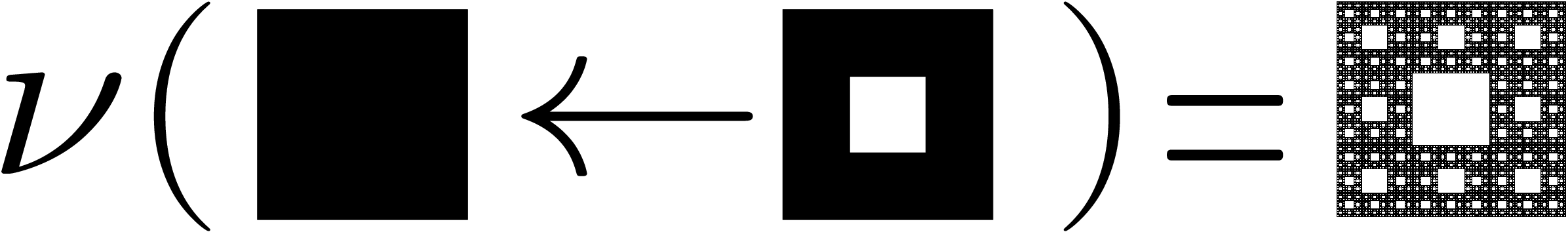}
    \caption{The Sierpinski carpet is the downward fixed point of the indicated algebra}\label{sierp}
\end{figure}
\hfill\\
The square metric space $M \otimes \blacksquare$ is the same as $\blacksquare$ except with the middle removed. There is an algebra $a \maps M \otimes \blacksquare \to \blacksquare $ given by the natural inclusion. As illustrated in Figure \ref{sierp}, the Sierpinski carpet is given by the downward fixed point $\nu$ applied to this algebra.
\begin{proposition}
The downward fixed point $\nu(\blacksquare \leftarrow M \otimes \blacksquare)$ is the Sierpinski carpet.
\end{proposition}
\begin{proof}
Because every morphism in the chain 
\[\hfill \blacksquare \leftarrow M \otimes \blacksquare \leftarrow M \otimes M \otimes \blacksquare \ldots \hfill\]
is an injection, its limit is the intersection \[\hfill\bigcap_{n=0}^{\infty} M^n \otimes \blacksquare \hfill\]
This infinite intersection is the usual definition of the Sierpinski carpet.

\end{proof}

We have seen that the Sierpinski carpet may be obtained in a more straightforward way than in \cite{noquez2021sierpinski} using a relatively terminal fixed point construction. Additionally, it can be shown directly that $a: M\otimes \blacksquare \to \blacksquare$ is in fact a corecursive algebra. Hence, by Corollary~\ref{coro:recursive}, we find that the Sierpinski carpet $\nu(a)$ is the terminal $M\otimes -$-coalgebra; this proof strategy avoids the need to construct specific paths through metric spaces, as is done in \cite{noquez2021sierpinski}. 

Other fractals may be generated as downward fixed points in a similar way; for example one can imagine that the Sierpinski triangle may constructed as a downward fixed point in a category of `triangular metric spaces'. 
\end{example}

\begin{example}\label{ex0}
In \cite{kori2023exploiting}, the authors state the Safety Problem. This problem may be rephrased in terms of the Galois connection 
\[ \hfill
\begin{tikzcd}
Post(F) \ar[r,bend left,"\mu_F"] \ar[r,phantom,"\bot"] & Pre(F) \ar[l,bend left,"\nu_F"]
\end{tikzcd}
\hfill \]
for a particular choice of $F$ and assuming that the set of initial states forms a post-fixed point.
\begin{definition}
A transition system is a triple $(S,I,\delta)$ where $S$ is a set of states, $I \subseteq S$, is a set of initial states, and $\delta \maps S \to \mathcal{P}(S)$ is a transitition relation. Here $\mathcal{P}(S)$ is the powerset of $S$ which is a complete lattice ordered by $\subseteq$. Let $F: \mathcal{P}(S) \to \mathcal{P}(S)$ be the monotone function defined by $F(X)=\bigcup_{x \in X} \delta(x)$ and suppose that $I$ is a post fixed point, i.e., $I \subseteq F(I)$. For a set $P \in \mathcal{P}(S)$, the \emph{Safety Problem} for $(I,P,S,F)$ asks if $\mu_F(I) \subseteq P$.
\end{definition}
The idea here is that $\mu_F(I)$ is the set of reachable states from $I$ and if $\mu_F(I) \subseteq P$, then we say that $I$ is $P$-safe. Now suppose that $P$ is a pre-fixed point $F(P) \subseteq P$. Then the adjunction of this paper says that 
\[\mu_F(I) \subseteq P \iff I \subseteq \nu_F(P)\]
While $\mu_F(I)$ represents the states reachable from $I$, $\nu_F(P)$ are the states which never go above $P$. In this case the adjunction suggests a strategy for verifying the Safety Problem. One may answer the Safety Problem by simultaneously unfolding $I$ and $P$ using $F$. In other words on the first step we check if $I \subseteq P$; if it is then we check $F(I) \subseteq P $ and $I \subseteq F(P)$. If either of those are false, then we know $I$ is not $P$-safe. If both are true then we continue to check $F^2(I) \subseteq P$ and $I \subseteq F^2(P)$. We continue this process indefinitely, checking to see if any of $F^n(I) \subseteq P$ and $I \subseteq F^n(P)$ are false. If we can't falsify any of these inclusions and we arrive at a fixed point (either $\mu_F(I)$ or $\nu_F(P)$), then we know that $I$ is $P$-safe.

A major limitation of this approach is that we require $I$ to be increasing and $P$ to be decreasing. In other cases, a different analysis will be necessary to verify safety. Regardless, we believe these ideas may be used to develop an effective algorithm for the Safety Problem.
\end{example}

\section{Preservation results}\label{sec:preservation}

In this section, we explore when functors preserve $\mu$ and $\nu$. To this end, we take inspiration from \cite{capretta2006recursive}, and focus on an adjoint situation equipped with a `step' $\theta$. 
Steps are a generalisation of distributive laws from
functors on one category to functors on two possibly different categories that are connected via an adjunction.
This requires the ingredients depicted in equation \eqref{eq:step}.
\begin{equation}\label{eq:step}\hfill
\begin{tikzcd}[cells={nodes={minimum size=0.8cm}}]
\C \arrow[loop left, "F"] \arrow[r, bend left, "L"{name=L}] & \D\arrow[loop right, "G"]\arrow[l, bend left, "R"{name=R}]
\arrow[phantom, "\vdash" marking, from=R, to=L]
\end{tikzcd}
\qquad \theta:LF\Rightarrow GL \hfill
\end{equation}
We note that such a $\theta$ comes equipped with its \emph{mate} $\theta^\flat:FR\to RG$ defined as the composite of
\[
\begin{tikzcd}
FR\arrow[r, "\eta"] & RLFR \arrow[r, "R\theta_R"] & RGLR\arrow[r, "RG\epsilon"] & RG
\end{tikzcd}
\]
(and indeed this mate correspondence is a bijection, as shown in \cite{kelly1974el2cat}). This situation covers a wide range of examples. Of particular interest are those cases where $\D$ is an Eilenberg-Moore category $\C^T$ or Kleisli category $\Kl(T)$ for a monad $T$ on $\C$. In these cases, the existence of a lifting or an extension $\bar F$ of an endofunctor $F:\C\to\C$ is equivalent to the existence of a step. \cite[Thm. 3]{rot2021steps} 

\begin{definition}
Consider the data of Scenario \eqref{eq:step}. $L$ extends to a functor $\bar L:\FCoalg\to \FCoalg[G]$ given by
\[\hfill
\begin{tikzcd}
FB\\
B\arrow[u,"b"]
\end{tikzcd}
\qquad\mapsto \qquad
\begin{tikzcd}
GLB\\
LFB\arrow[u, "\theta"]\\
LB\arrow[u, "Lb"]
\end{tikzcd}
\hfill
\]
Similarly, $R$ extends to a functor $\bar R:\FAlg[G]\to \FAlg$ given by
\[\hfill
\begin{tikzcd}
GA\arrow[d, "a"]\\
A
\end{tikzcd}
\qquad\mapsto\qquad
\begin{tikzcd}
FRA\arrow[d, "\theta^\flat"]\\
RGA\arrow[d, "Ra"]\\
RA
\end{tikzcd}
\hfill
\]
\end{definition}
These functors $\bar L$ and $\bar R$ satisfy something akin to an adjoint relationship. Before stating this relationship, we recall the following (`useful') lemma:
\begin{lemma}\label{lemma:useful}
If $\theta:LF\to GL$ is a step, with mate $\theta^\flat$, the following two squares commute:
\[
\hfill
\begin{tikzcd}
F\arrow[d, "\eta_F"] \arrow[r, "F\eta"] & FRL\arrow[d, "\theta^\flat_L"] & LFR\arrow[d, "\theta_R"] \arrow[r, "L\theta^\flat"] & LRG\arrow[d, "\epsilon_G"]\\
RLF\arrow[r, "R\theta"] & RGL & GLR\arrow[r, "G\epsilon"]& G
\end{tikzcd}
\hfill
\]
\end{lemma}
See e.g. \cite{rot2021steps} for a proof.

\begin{lemma}\label{lemma:baradj}
Let $b:B\to FB$ be an $\F$-coalgebra, and $a:GA\to A$ a $G$-algebra. The natural isomorphism $\Hom_\D(LB, A)\cong \Hom_\C(B, RA)$ restricts to a natural isomorphism
\[
\hfill
\Hylo[G](\bar Lb, a)\cong \Hylo(b, \bar Ra)
\hfill
\]
\end{lemma}

\begin{proof}
Fix a coalgebra-to-algebra morphism $f:\bar Lb\to a$. Consider $f$'s transpose $\tilde f = Rf\circ\eta$ along the adjunction. We claim that $\tilde f$ is a coalgebra-to-algebra morphism $b\to \bar Ra$. This can be seen in the following diagram:
\[
\hfill
\begin{tikzcd}
FRLB\arrow[rr, "FRf"]\arrow[dr, "\theta^\flat_L"] && FRA\arrow[d, "\theta^\flat"]\\
FB\arrow[u, "F\eta"]\arrow[dr, "\eta_F"]& RGLB\arrow[r, "RGf"] & RGA\arrow[dd, "Ra"]\\
& RLFB\arrow[u, "R\theta"]\\
B\arrow[uu, "b"] \arrow[r, "\eta"] & RLB\arrow[u, "RLb"] \arrow[r, "Rf"] &RA
\end{tikzcd}
\hfill
\]
Here, the bottom right square is the ca-morphism square for $f$; the top right is a naturality square for $\theta^\flat$; the top left is given by Lemma \ref{lemma:useful}; and the bottom left is naturality for $\eta$. The outside of the square is a ca-morphism square for $\tilde f$. 

On the other hand, let $g :b\to \bar Ra$ be a ca-morphism. We claim that its transpose is again a ca-morphism $\bar Lb\to a$. This is completely dual to the previous case; but for completeness, it can be seen in the following diagram:
\[
\hfill
\begin{tikzcd}
GL(B)\arrow[rr, "GLg"] &&GLRA\arrow[d, "G\epsilon"]\\
LF(B)\arrow[u, "\theta"]\arrow[r, "LFg"] & LFRA\arrow[ur, "\theta"]\arrow[d, "L\theta^\flat"]&GA\arrow[dd, "a"]\\
&LRGA\arrow[d, "LRa"]\arrow[ur, "\epsilon"]\\
L(B)\arrow[uu, "b"] \arrow[r, "Lg"] & LRA\arrow[r, "\epsilon"] & A
\end{tikzcd}
\hfill
\] 
\end{proof}

In \cite{capretta2006recursive}, it was shown that $\bar L$ preserves recursive coalgebras, and (dually) $\bar R$ preserves corecursive algebras. This now follows directly from the above lemma; however, we can obtain the stronger result that $\bar L$ commutes with the induced monad $\nu\mu$, and $\bar R$ commutes with the induced comonad $\mu\nu$. 
\begin{theorem}\label{thm:steppreserve}
Consider an adjoint situation as in \eqref{eq:step}. Let $b:B\to FB$ be an $\F$-coalgebra, and $a:GA\to A$ a $G$-algebra. 
\begin{enumerate}[label = (\roman*)]
\item $\nu\mu(\bar Lb) = \bar L(\nu\mu(b))$
\item $\mu\nu(\bar Ra) = \bar R(\mu\nu(a))$
\end{enumerate}
\end{theorem}
\begin{proof}
We only prove (i), since (ii) follows by duality. Let $b$ be a (fixed) $\F$-coalgebra, and $a$ a $G$-algebra. By Remark \ref{rem:flip}, we know that $\nu\mu(b) = \mu(b)^{-1}$, and $\nu\mu(\bar Lb) = \mu(\bar Lb)^{-1}$; hence, it suffices to show
\[
\mu(\bar Lb) = (\bar L(\mu (b)^{-1}))^{-1}
\]
Using Lemma \ref{lemma:baradj}, we have the following chain of equivalences, natural in $a$:
\begin{align*}
\Hylo[G](\bar Lb,a)&\cong \Hylo(b, \bar Ra)\\
&\cong \Hom_{\FAlg}(\mu(b), \bar Ra)\\
&\cong \Hylo(\mu(b)^{-1}, \bar Ra)\\
&\cong \Hylo[G](\bar L(\mu(b)^{-1}),a)\\
&\cong \Hom_{\FAlg[G]}(\bar L(\mu(b)^{-1})^{-1}, a)
\end{align*}
\end{proof}
This general theorem can be used to prove preservation in various specific circumstances. 
\begin{corollary}\label{coro:monad}
Let $T:\C\to\C$ be a monad, let $F:\C\to \C$ be an endofunctor. Write $j\dashv U$ for the Kleisli adjunction of the monad, and $T\dashv |-|$ for the Eilenberg-Moore adjunction. 
\begin{enumerate}[label = (\roman*)]
\item Assume that $\F$ extends to a functor $\bar F:Kl(T)\to Kl(T)$ with $\bar Fj = jF$. Then $j$ commutes with $\mu$ and $U$ commutes with $\nu$.
\item Assume that $\F$ extends to a functor $\bar F:EM(T)\to EM(T)$ with $\bar FT = TF$. Then $T$ commutes with $\mu$ and $|-|$ commutes with $\nu$. 
\end{enumerate}
\end{corollary}

\begin{proof}
Both of these are instances of adjoint situation \eqref{eq:step} with the step given by identities, and hence the statement follows immediately from Theorem~\ref{thm:steppreserve}.
\end{proof}

\subsection*{$\mu$ and $\nu$ Coincide in a Dagger Category}

When coalgebras for a polynomial functor $F:\Set \to \Set$ are interpreted as $\F$-shaped automata, the initial $\F$-algebra serves as finite trace semantics and the (weakly) terminal $\F$-coalgebra gives an infinite trace semantics~\cite{jacobs04a,hasichsok:2007}. When $\F$ is no longer a $\Set$-functor this interpretation breaks down. For example if $F: \Rel\to \Rel$, where $\Rel$ is the category of sets and relations, then the initial algebra and terminal coalgebra coincide \cite{smyth1982category}. In \cite{karvonen2019way}, it is shown that this holds more generally in any dagger category. With this coincidence, the initial algebra/final coalgebra gives a finite trace semantics instead of an infinite trace semantics. To obtain a semantics for infinite traces, Urabe and Hasuo construct an object which is weakly terminal among coalgebras and define the infinite trace semantics as the maximal map into this object \cite{hasuo2018coalgebraic}. Note that the limit colimit coincidence causes no issues when $\mu(c)$ is interpreted as a semantic object for $c$. However, a generalized limit colimit coincidence also holds for the fixed points generated by $\mu$ and $\nu$.
\begin{definition}
A dagger category $(\C,\dag)$ is a category equipped with an identity on objects functor $\dag: \C \to \C^{\op}$ such that $\dag^2=id$.\end{definition}
\begin{theorem}
    Suppose that $(\C,\dag)$ is a dagger category with limits and colimits of countable chains and $F: \C \to \C$ is a dagger functor preserving such limits and colimits. Then  there is an isomorphism
    \[ \hfill \mu(c)^\dag \cong \nu(c^\dag)\hfill \]
    for each coalgebra $c$. Dually, for each algebra $a$, there is an isomorphism $\nu(a)^\dag \cong\mu(a^\dag)$.
\end{theorem}

This theorem may be viewed as a special case of Theorem \ref{thm:steppreserve} but it is simpler to use the construction as a (co)limit.

\begin{proof}
For a coalgebra $X \xrightarrow{c} FX$ we have
\begin{align*}
    \nu(c^{\dag}) & \cong \lim(X \xleftarrow{c^{\dag}} FX \xleftarrow{Fc^{\dag}} F^2X \leftarrow \ldots ) \\
    &\cong \colim_{C^{\op}} (X \xleftarrow{c^{\dag}} FX \xleftarrow{Fc^{\dag}} F^2X \leftarrow \ldots ) \\
    &\cong \colim(X \xrightarrow{c} FX \xrightarrow{Fc} F^2X \to \ldots )^{\dag} \\
    &\cong \mu(c)^{\dag}
\end{align*}
The second isomorphism is because limits in $\C$ are colimits in $\C^{op}$ and the third isomorphism is because $\dag$ preserves colimits because it is an equivalence. A similar proof holds for the dual statement.
\end{proof}

\section{Conclusion}\label{sec:conclusions}

In this paper, we have demonstrated that the relative fixed points of a functor induce an adjunction between $F$-coalgebras and $F$-algebras. Based on this observation, we have studied various instances of relative fixed points of functors. In some of these instances, such as in the case of the Sierpinski carpet, the fixed points from these functors have previously been presented as initial algebras or final coalgebras. In other instances, the fixed points are novel, as is the case with the general description of relative fixed points of polynomial functors in Proposition~\ref{prop:polychar} and concrete instances of this description in subsequent examples. 

Relative fixpoints provide a fresh perspective on ca-morphisms. Previous work has mostly focused on cases where there is a unique ca-morphism, via the notions of recursive algebras and corecursive coalgebras \cite{capretta2009corecursive}. However, in \cite{hauhs2015scientific}, the authors argue that ca-morphisms also hold interest when they are not unique. Using examples in probability, dynamical systems, and game theory, the authors show how non-unique ca-morphisms often represent solutions to problems in these disciplines. This gives us hope that relative fixed points and the results we have proven about them may be useful in these applications as well. In particular, in future work we will develop the algorithm suggested in Example~\ref{ex0} and expand its capabilities to solve a wider range of problems.

Another direction of future work is to understand the connection between relatively terminal coalgebras and coequations. As discussed in Section \ref{sec:constructions}, $\nu (a)$ may be thought of as a `cofree solution of the coequation $a$'. As such, studying $\nu$ may yield new insights into this class of `(flatly) corecursive coequations', and the kind of properties that may be defined by such. 

\bibliography{cs.bib}

\end{document}